\documentclass{birkjour}
 \newtheorem{theorem}{Theorem}[section]
 
 \newtheorem{lemma}[theorem]{Lemma}
 
 \theoremstyle{definition}
 
 \theoremstyle{remark}
 \newtheorem{rem}[theorem]{Remark}
 \newtheorem{example}{Example}
 \newtheorem{problem}{Problem}
 \numberwithin{equation}{section}

\usepackage{hyperref}

\usepackage{amsmath,amsthm}
\usepackage{amssymb}

\usepackage{enumitem}

\usepackage{graphicx}

\usepackage{caption}
\usepackage{subcaption}

\usepackage[T1]{fontenc}
\usepackage{amsfonts}
\usepackage{float}
\usepackage{cancel}
\begin{document}

%-------------------------------------------------------------------------
% editorial commands: to be inserted by the editorial office
%
%\firstpage{1} \volume{228} \Copyrightyear{2004} \DOI{003-0001}
%\seriesextraline{This is the Concrete Title of this Book\br H.E. R and S.T.C. W, Eds.}
%\firstpage{1}

\title[The number of inscribed and circumscribed graphs]
 {The number of inscribed and circumscribed\\ graphs of a convex polyhedron}

%----------Author 1
\author{Yagub N. Aliyev}

\address{%
School of IT and Engineering\\ 
ADA University\\
Ahmadbey Aghaoglu str. 61 \\
Baku 1008, Azerbaijan}

\email{yaliyev@ada.edu.az}

\thanks{This work was supported by ADA University Faculty Research and Development Fund.}



\date{January 1, 2004}
%\dedicatory{}
%%% ----------------------------------------------------------------------

\begin{abstract}
In the paper we prove that the number of graphs inscribed into graph of a convex polyhedron and circumscribed around another graph does not exceed 4. For this we first studied Poncelet type problem about the number of convex $n$-gons inscribed into one convex $n$-gon and circumscribed around another convex $n$-gon. It is proved that their number is also at most 4. This contrasts with Poncelet type porisms where usually infinitude of such polygons is proved, provided that one such polygon already exists. An inequality involving ratio of lengths of line segments is used. Alternative way of using Maclaurin-Braikenridge's conic generation method is also discussed. Properties related to constructibility with straightedge and compass are also studied. A new proof, based on mathematical induction, of generalized Maclaurin- Braikenridge's theorem is given. We also gave examples of regular polygons and a polyhedron for which number 4 is realized.
\end{abstract}

%%% ----------------------------------------------------------------------
\maketitle
%%% ----------------------------------------------------------------------
%\tableofcontents
\section{Introduction} Polyhedra can be defined in various generality as geometric objects consisted of vertices (points), edges (line segments), and faces (polygons). Regular (Platonic), semi-regular (Archimedean), regular star polyhedra (Kepler–Poinsot), and other types of polyhedra frequently appear in mathematics, biology, crystallography, physics, etc. Convex polyhedra which are defined as convex hull of finitely many points in space play important role in geometry and its applications. Some of these polyhedra are obtained by choosing its vertices on the edges of another polyhedron and this process can be repeated indefinitely. For example, midpoints of the edges of tetrahedron give the vertices of an octahedron, whose edge midpoints give the vertices of an icosahedron. The midpoints of icosahedron and dodecahedron both give the vertices of icosidodecahedrons. The literature about history, classification, and applications of these polyhedra is extensive (see \cite{crom} and its references). In the current paper a general result is proved about the graphs of such nested polyhedra.

Let $\alpha$ be a graph (vertices and edges) also called 1-skeleton (\cite{grun}, p. 138) of a convex polyhedron. On each edge of $\alpha$ choose a point which does not coincide with the vertices of $\alpha$. If the points on each polygonal face of $\alpha$ are connected to form a convex polygon, then these polygons together form a graph, which we denote by $\beta$ (see Figure 1). In particular, if the number of faces meeting at each vertex of $\alpha$ is 3 as in cube (see Figure 1), dodecahedron (see Figure 15), and tetrahedron (see Figure 13), then $\beta $ can also be interpreted as a convex polyhedron obtained by truncation of the vertices of the polyhedron of $\alpha$ (see \cite{coxeter1}, Chapter 8). This means that in this case the convex hull of the vertices of $\beta$ form a convex polyhedron whose edges coincide with the edges of $\beta$ (see Figure 1). But if the number of faces meeting at some vertex of $\alpha$ is greter than 3 as in octahedron (see Figure 2) and icosahedron (see Figure 16), then this interpretation is not possible, and therefore $\beta$ is just a graph.
We say $\beta$ is inscribed into $\alpha$ and similarly, $\alpha$ is circumscribed around $\beta$. Repeating this process for $\beta$, we obtain its inscribed graph $\gamma$. The vertices of $\gamma$ are chosen on the edges of $\beta$ one on each. Part of the edges of $\gamma$ form convex polygons on the faces of $\alpha$. The remaining part of the edges of $\gamma$ is not important to specify as they do not play any role in our considerations but for clarity we can assume them to connect only those 2 vertices of $\gamma$, which belong to 2 edges of $\beta$ having common vertex of $\beta$ and subtending angles on faces of $\alpha$ with common vertex of $\alpha$. At each vertex of $\beta$ 4 edges and 4 faces are meeting and therefore interpretation of $\gamma$ as a graph of a convex polyhedron obtained by truncation of the vertices of the polyhedron of $\beta$ is not always possible. Suppose now that graphs $\alpha$ and $\gamma$ are fixed. Is it possible to find graphs different from $\beta$ which are also inscribed into $\alpha$ and circumscribed around $\gamma$, and if yes, then what is the maximum of their number? We will answer the second question by proving that this number can not exceed 4 and give an example of a convex polyhedron for which 4 such graphs exist (see Figure 2).

\begin{figure}
\centering
\begin{minipage}{.5\textwidth}
  \centering
  \includegraphics[width=1\linewidth]{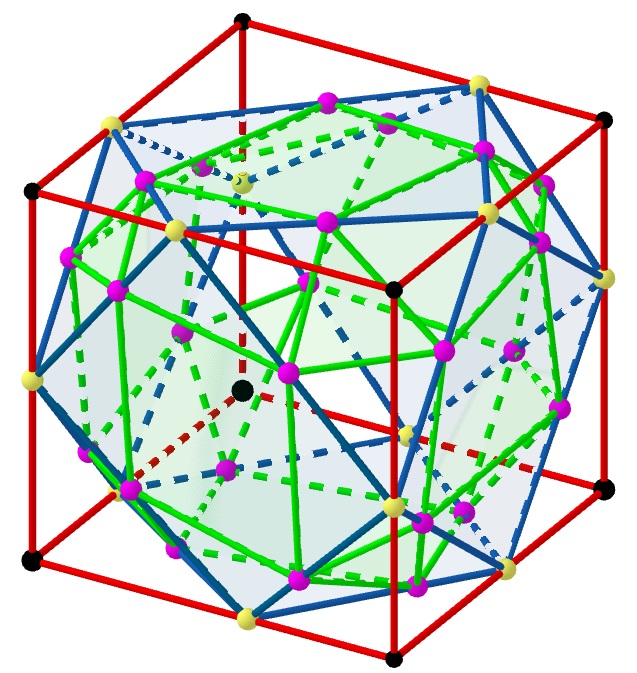}
  \captionof{figure}{Nested graphs of a convex polyhedron (\url{https://www.geogebra.org/3d/kuqsez2b}).}
  \label{fig1}
\end{minipage}%
\begin{minipage}{.5\textwidth}
  \centering
  \includegraphics[width=1\linewidth]{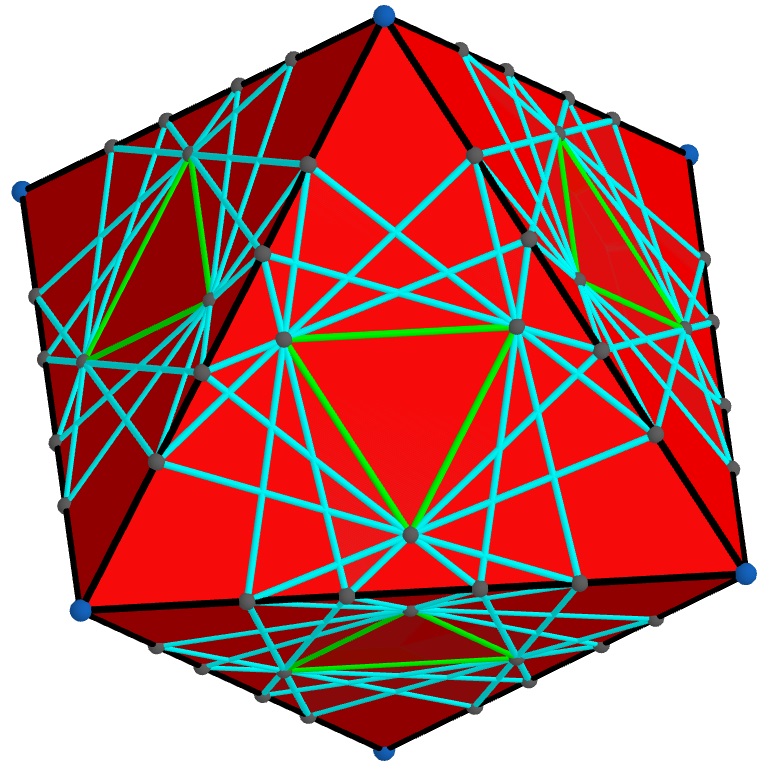}
  \captionof{figure}{Icosahedron with 4 inscribed graphs (\url{https://www.geogebra.org/3d/jwhddbeg}).}
  \label{fig2}
\end{minipage}
\end{figure} 

Note that all the edges of $\beta $ and all the vertices of $\gamma$ are on the faces of $\alpha$. This means that the polygonal face of $\beta$ inscribed into a face of $\alpha$ is also circumscibed around the polygonal face of $\gamma$ lying on $\alpha$. All these faces are convex polygons and therefore we need to first solve the problem for the plane which has its own history and it is discussed in the next section.

\section{Planar problem}
Poncelet in his introduction to famous porism theorem, studied momentarily the variable polygons having all but one of their vertices on sides of a given polygon and all the sides passing
through the vertices of another polygon. He shows that the free vertex describes a conic (\cite{poncelet}, Tome I, p. 308-314). Poncelet cites Brianchon \cite{brian}, who cites Maclaurin \cite{mac} and Braikenridge 
\cite{braik} for the case of triangles $(n = 3)$ (see \cite{mills} for the history). Poncelet started to work on these questions in 1813 when he was a prisoner of war in Russia.
In Poncelet porisms infinitude of polygons inscribed into one ellipse and and circumscribed around another ellipse, given that one such polygon exists, is proved \cite{poncelet}, \cite{poncelet2} (see also \cite{centina}, \cite{centina2}, \cite{halb}). This result inspired discovery of many other porisms commonly known as Poncelet type theorems. In Steiner porism infinitude of chains of tangent circles all of which are internally tangent to one circle and externally tangent to another circle, again provided that one such chain exists, is studied (see\cite{pedoe}, p.98; \cite{coxeter}, sect. 5.8). In Emch's porism Steiner's porism is generalized to chain of intersecting circles, where these intersections are on another circle \cite{emch} (see also \cite{avk}). In Zig-zag type porisms infinitude of equilateral polygons with vertices taken alternately on two circles (or lines), is studied \cite{bottema}, \cite{hras}, \cite{csikos}. Money-Coutts theorem (see \cite{tabach} and its references) is about chain of tangent circles alternately inscribed into angles of a triangle. Generalizations and connections of these results with each other, were studied in \cite{bauer}, \cite{protasov}, \cite{drag}. Similar generalizations for the space with rings of tangent spheres were discussed in \cite{koll}, \cite{coxeter0}. Surprisingly, the original configuration of Maclaurin about polygons inscribed into one polygon and circumscribed around another polygon did not attract much attention after these powerful generalizations. In the current section we will study in detail the case of convex polygons inscribed into one convex polygon and circumscribed around another one. The results are surprising as we will show that the maximal number of such polygons is not dependent on the number of sides of the polygons. At the end of Section 4 the case of generalized polygons, considered as collections of vertices or lines is also considered. Possibility of construction of such regular polygons with Euclidean instruments (straightedge and compass) is disscussed in Appendix A. This construction algorithm was very useful for drawing diagrams in the current paper all of which, except the last two, are created using the website of GeoGebra. For completeness we also provided an independent proof for generalized Maclaurin-Braikenridge's conic generation method in Appendix B. 

Suppose a convex polygon $A_0 A_1 A_2\ldots A_{n-1}$ $(n\ge 3$) is given (see Figure 3). Choose a point on each side of this polygon. Denote these points by $B_i\in A_i A_{i+1}$, where $B_i$ is between $A_i$ and $A_{i+1}$, and $B_i\ne A_i,A_{i+1}$ $(i=0,1,\ldots,n-1)$. For such polygons we say that $B_0 B_1 B_2\ldots B_{n-1}$ is inscribed into $A_0 A_1 A_2\ldots A_{n-1}$ and $A_0 A_1 A_2\ldots A_{n-1}$ is circumscribed around $B_0 B_1 B_2\ldots B_{n-1}$. In the same way choose on each side of polygon $B_0 B_1 B_2\ldots B_{n-1}$ a point $C_i\in B_i B_{i+1}$ ($C_i$ is between $B_i$ and $B_{i+1}$, and $C_i\ne B_i,B_{i+1}$ for $i=0,1,\ldots,n-1$) and draw polygon $C_0 C_1 C_2\dots C_{n-1}$. Now fix polygons $A_0 A_1 A_2\ldots A_{n-1}$ and $C_0 C_1 C_2\dots C_{n-1}$. We want to find the maximum of the number of polygons like $B_0 B_1 B_2\ldots B_{n-1}$, which are inscribed into $A_0 A_1 A_2\ldots A_{n-1}$ and circumscribed around $C_0 C_1 C_2\dots C_{n-1}$, possibly with the indices shifted so that $C_i\in B_{k+i} B_{k+i+1}$ $(i=0,1,\ldots,n-1)$ for some fixed $k$. We will prove that this number is 4. Our strategy is to give an example where number 4 is realized and then prove that this number can not exceed 4.

\begin{figure}
\centering
\begin{minipage}{.5\textwidth}
  \centering
  \includegraphics[width=.8\linewidth]{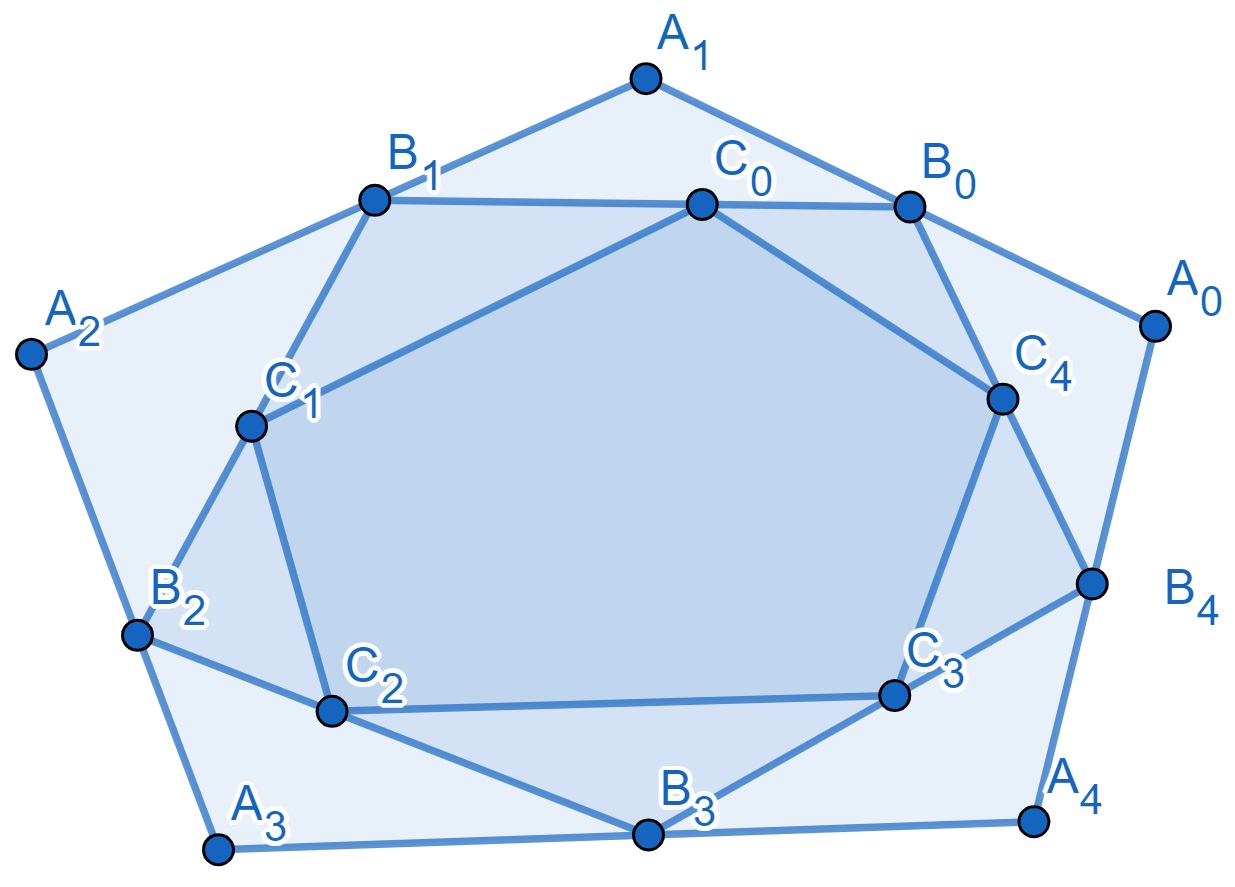}
  \captionof{figure}{Inscribed and circumscribed convex polygon.}
  \label{fig3}
\end{minipage}%
\begin{minipage}{.5\textwidth}
  \centering
  \includegraphics[width=1\linewidth]{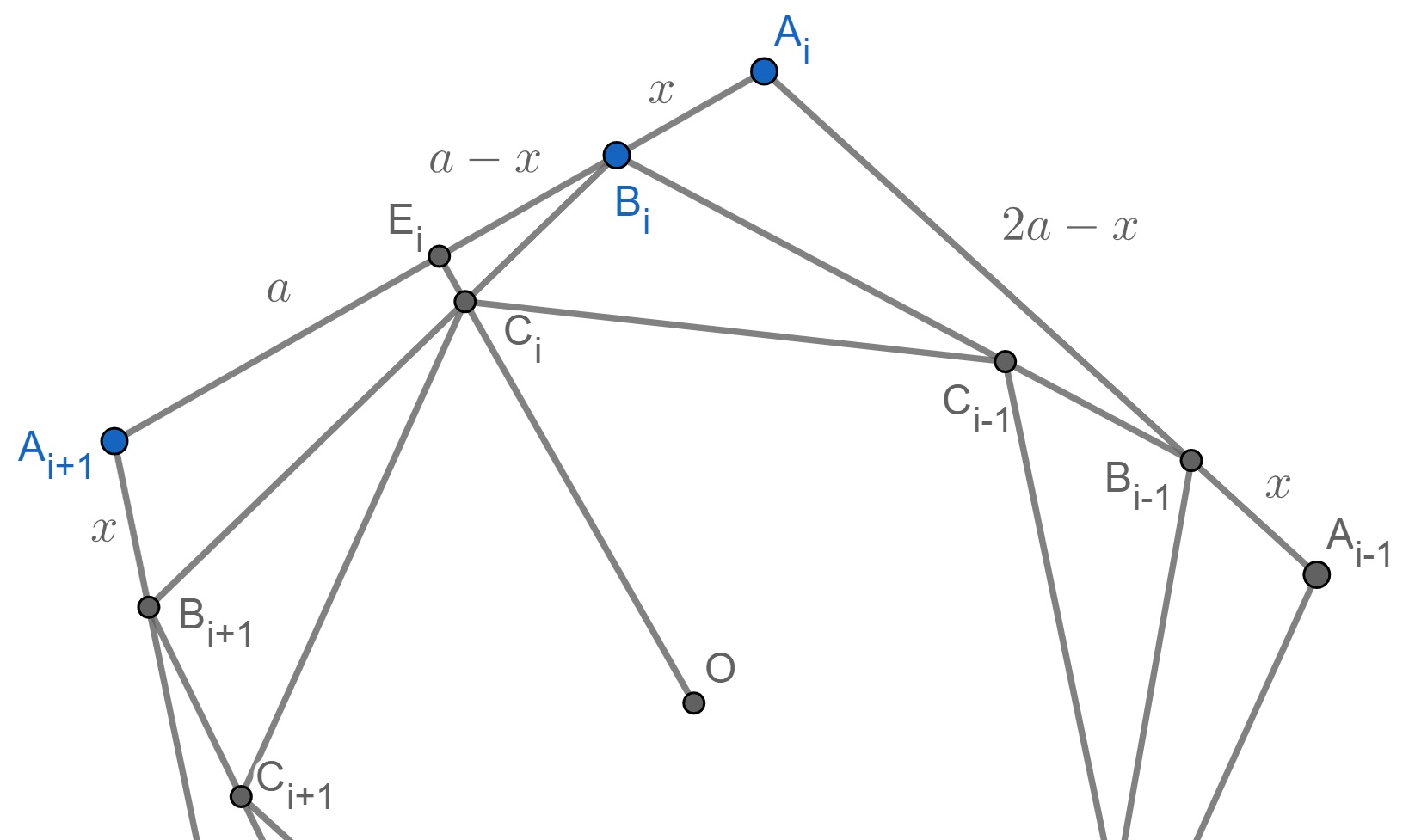}
  \captionof{figure}{Example (Regular polygon).}
  \label{fig4}
\end{minipage}
\end{figure}

The key ingredients of the proof will be the following examples and the lemma, which are interesting on their own.

\begin{example}
Let $A_0 A_1 A_2\ldots A_{n-1}$ be a regular polygon with with side length equal to $2a$ and center $O$. Let $E_i$ be the midpoints of $A_i A_{i+1}$ $(i=0,1,\ldots,n-1)$. Let $B_i\in A_i E_{i}$ and denote $|A_iB_i|=x$ $(i=0,1,\ldots,n-1)$. Then $|E_iB_i|=a-x$ (see Fig. 4). Denote intersection point of line segments $OE_i$ and $B_iB_{i+1}$ by $C_i$. Then
$$
|C_iE_i|=\frac{x(a-x)\sin{\frac{2\pi}{n}}}{2a-2x\sin^2{\frac{\pi}{n}}}.
$$
Denote $
f(x)=\frac{x(a-x)\sin{\frac{2\pi}{n}}}{2a-2x\sin^2{\frac{\pi}{n}}}.
$
One can check that $$f\left(\frac{a}{2}\right)=f\left(\frac{a}{1+\cos^2{\frac{\pi}{n}}}\right)=\frac{a\sin{\frac{2\pi}{n}}}{4\left(1+\cos^2{\frac{\pi}{n}}\right)}.$$ 
Therefore the two regular polygons $B_0 B_1 B_2\ldots B_{n-1}$ corresponding to $x=\frac{a}{2}$, $x=\frac{a}{1+\cos^2{\frac{\pi}{n}}}$ and their reflections with respect to line $OE_i$ form 4 polygons, which are inscribed into $A_0 A_1 A_2\ldots A_{n-1}$ and circumscribed around $C_0 C_1 C_2\dots C_{n-1}$. Note that points $B_i$ corresponding to $x=\frac{a}{2}$ and $x=\frac{a}{1+\cos^2{\frac{\pi}{n}}}$ are constructible with compass and straightedge, provided that regular polygon $A_0 A_1 A_2\ldots A_{n-1}$ is given (see Appendix A).

\begin{figure}
\centering
\begin{minipage}{.5\textwidth}
  \centering
  \includegraphics[width=1\linewidth]{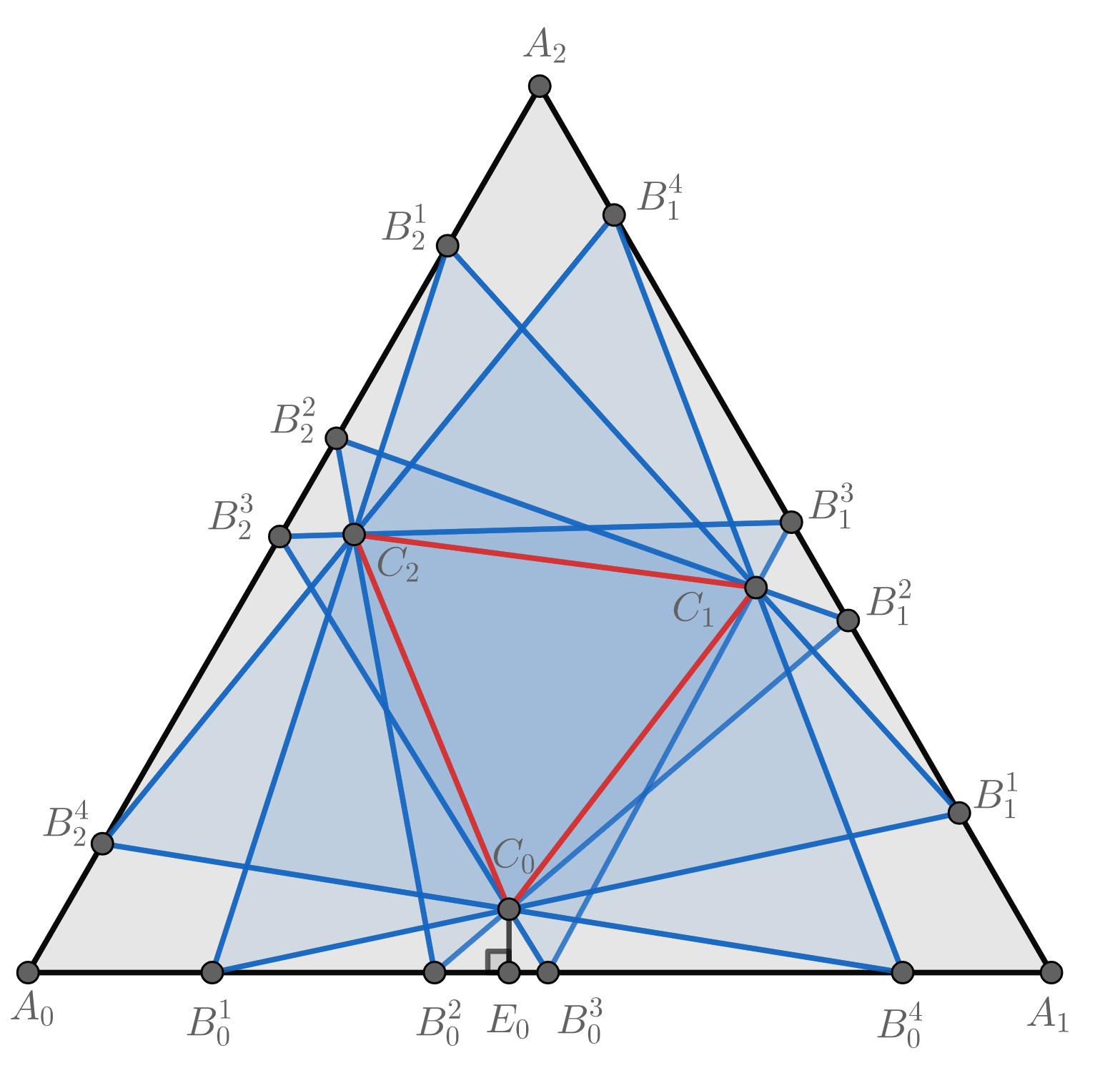}
  \captionof{figure}{$n=3$.}
  \label{fig5}
\end{minipage}%
\begin{minipage}{.5\textwidth}
  \centering
  \includegraphics[width=1\linewidth]{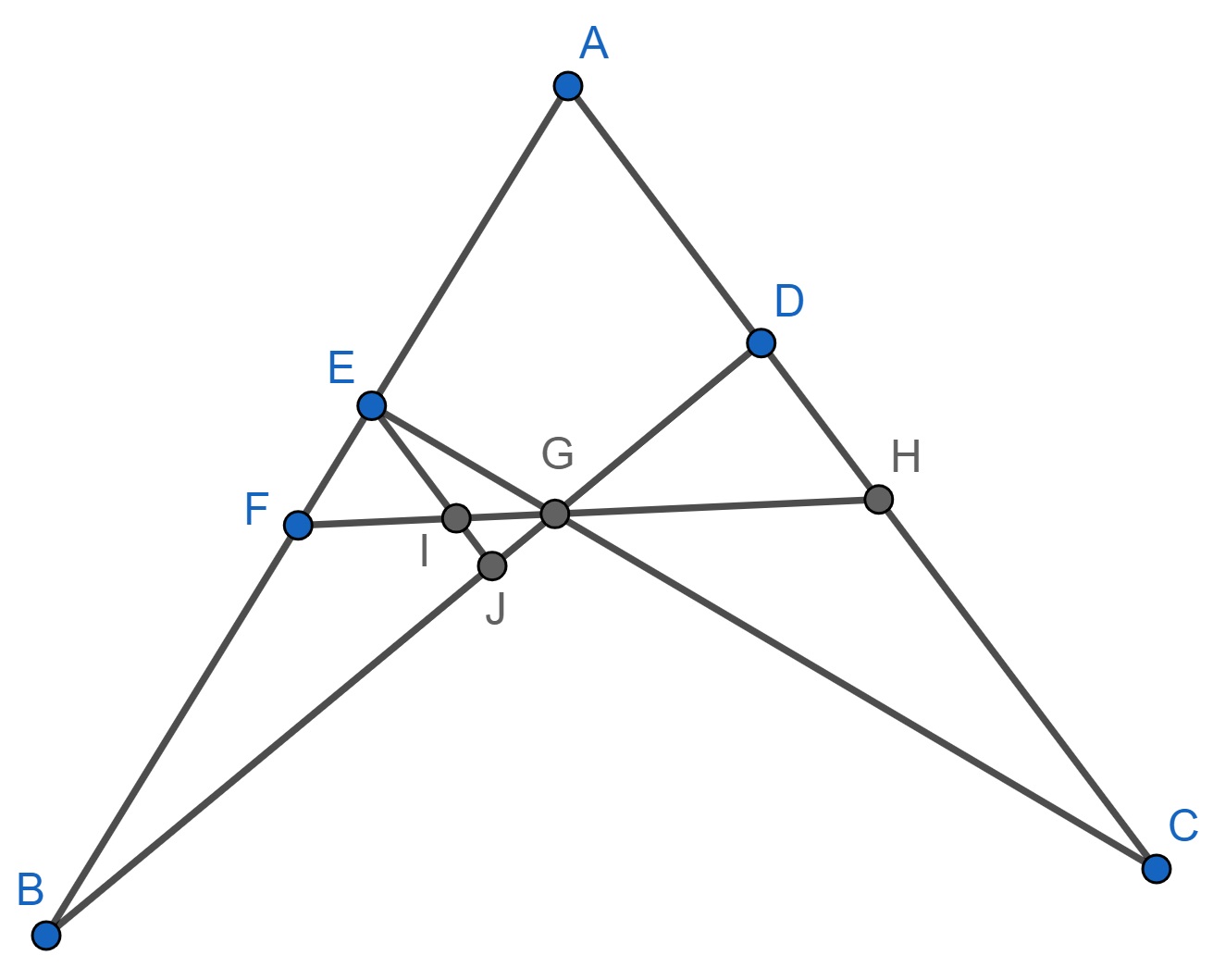}
  \captionof{figure}{Inequality $\frac{|BF|}{|FE|}>\frac{|DH|}{|HC|}$.}
  \label{fig6}
\end{minipage}
\end{figure}

\end{example}

\begin{example}
We will show how to find other exampes of 4 polygons inscribed into one and circumscribed around another convex polygon. Without loss of generality we can assume that $|A_iA_{i+1}|=1$ ($i=0,1,\ldots,n-1$). We use notations similar to those from Example 1 (see Figure 5 for case $n=3$). Denote $|A_iB_i^1|=x$, $|A_iB_i^2|=y$, $|A_{i+1}B_i^3|=z$, $|A_{i+1}B_i^4|=t$, $|A_iE_i|=a$, where $E_i$ is the base of perpendicular from point $C_i$ to line $A_iA_{i+1}$ for $i=0,1,\ldots,n-1$. We find that
$$
|C_iE_i|=f\left(a,x\right):= \left(a-x\right)\cdot \left(\frac{\sin\left(\frac{2\cdot \pi}{n}\right)}{\frac{1}{x}-2\cdot \left(\sin\left(\frac{\pi}{n}\right)\right)^{2}}\right).
$$
By solving equation $f(a,y)=f(a,x)$ for $y$ we obtain
$$
y=\frac{a-x}{1-2x\sin^2{\frac{\pi}{n}}}.
$$
So, $z,t$ are the solutions of equation $f(1-a,u)=f(a,x)$ for $u$. By changing $x$ and $a$, we can get infinitely many examples for each $n\ge3$ (see Figure 5 and \url{https://www.geogebra.org/calculator/exxbufdu}). It would be useful to find $x,y,z,t$ that can work both for regular $n$-gon and regular $m$-gon, when $m\ne n$. But we can show that it is impossible. Indeed, otherwise we would obtain
$$
a=x+y-2xy\sin^2{\frac{\pi}{n}},\ 1-a=z+t-2zt\sin^2{\frac{\pi}{n}},
$$
$$
a'=x+y-2xy\sin^2{\frac{\pi}{m}},\ 1-a'=z+t-2zt\sin^2{\frac{\pi}{m}},
$$
where $a'=|A_iE_i|$ is analogous distance for the $m$-gon. By adding the corresponding equalities we would obtain
$$
x+y+z+t-2(xy+zt)\sin^2{\frac{\pi}{n}}=1,\ x+y+z+t-2(xy+zt)\sin^2{\frac{\pi}{n}}=1,
$$
which is impossible because $\sin^2{\frac{\pi}{n}}\ne\sin^2{\frac{\pi}{m}}$ whenever $n\ne m$.
\end{example}

\begin{lemma}
Let cevians $BD$ and $CE$ of $\triangle ABC$ intersect at point $G$. A line through point $G$ intersects line segments $BE$ and $CD$ at points $F$ and $H$. Then $\frac{|BF|}{|FE|}>\frac{|DH|}{|HC|}$.
\end{lemma}

\begin{proof}
Let the line through point $E$ and parallel to line $AC$ intersect line segments $FG$ and $BG$ at points $I$ and $J$ (see Figure 6). By Menelaus' theorem $$\frac{|BF|}{|FE|}\cdot \frac{|EI|}{|IJ|}\cdot \frac{|JG|}{|GB|}=1.$$
By similarity of $\triangle EIG$ and $\triangle CHG$, $\triangle JIG$ and $\triangle DHG$, $\triangle EGJ$ and $\triangle CGD$, we obtain $\frac{|EI|}{|IJ|}=\frac{|CH|}{|HD|}$. Finally, noting $|JG|<|GB|$, we obtain $$\frac{|BF|}{|FE|}=\frac{|IJ|}{|EI|}\cdot \frac{|GB|}{|JG|}=\frac{|DH|}{|HC|}\cdot \frac{|GB|}{|JG|}>\frac{|DH|}{|HC|}.$$
\end{proof}

\begin{theorem}
The number of convex polygons $B_0 B_1 B_2\ldots B_{n-1}$, that can be inscribed into a given convex polygon $A_0 A_1 A_2\ldots A_{n-1}$ and circumscribed around around another convex polygon $C_0 C_1 C_2\dots C_{n-1}$ is at most 4.
\end{theorem}

\begin{proof}
We saw in Example 1 (see also more general Example 2) that the case with 4 polygons is possible. Suppose on the contrary that there is a configuration with at least 5 polygons. Denote 5 of them by $B^j_0 B^j_1 B^j_2\ldots B^j_{n-1}$ $(j=1,2,\ldots,5)$ so that $|A_0B^j_0|<|A_0B^{j+1}_0|$, for $j=1,2,3,4$ (see Figure 7). Immediately, $|A_iB^j_i|<|A_iB^{j+1}_i|$, for all $i=0,1,\ldots,n-1$ and $j=1,2,3,4$. Consider the first two of the polygons: $B^1_0 B^1_1 B^1_2\ldots B^1_{n-1}$ and $B^{2}_0 B^{2}_1 B^{3}_2\ldots B^{3}_{n-1}$.
Then the vertices of $C_0 C_1 C_2\dots C_{n-1}$ are $C_i=B^1_iB^1_{i+1}\cap B^{2}_iB^{2}_{i+1}$ $(i=0,1,\ldots,n-1)$, where all the indices are taken modulo $n$. Note that $B^{3}_0 B^{3}_1$ can not pass through $C_0$. Indeed, otherwise $B^{3}_i B^{3}_{i+1}$ pass through $C_i$ for all $i=0,1,\ldots,n-1$, and by Lemma 2.1,
$$
\frac{|B^1_iB^2_i|}{|B^2_iB^3_i|}>\frac{|B^1_{i+1}B^2_{i+1}|}{|B^2_{i+1}B^3_{i+1}|},
$$
for all $i=0,1,\ldots,n-1$, which is impossible, because it implies $
\frac{|B^1_0B^2_0|}{|B^2_0B^3_0|}>\frac{|B^1_0B^2_0|}{|B^2_0B^3_0|}.
$
So, by convexity of $A_0 A_1 A_2\ldots A_{n-1}$ and $C_0 C_1 C_2\dots C_{n-1}$, $B^{3}_0 B^{3}_1$ passes through $C_1$. This means that $C_i\in B^{3}_{k+i} B^{3}_{k+i+1}$ $(i=0,1,\ldots,n-1)$ for $k=-1$. We can show that $k=0$ and $k=-1$ are the only possible choices of $k$ for which $C_i\in B^{j}_{k+i} B^{j}_{k+i+1}$ ($i=0,1,\ldots,n-1$, and $j=1,2,3,4$).
Indeed, if $n>3$, then all sets $$\mathbb{M}_i=\triangle A_iA_{i+1}A_{i+2}\cap \triangle A_{i-1}A_iA_{i+1}\ (i=0,1,\ldots,n-1)$$ are disjoint (for $\mathbb{M}_i$ only the interior regions of the triangles are considered) and $C_i\in \mathbb{M}_i$. If $n=3$, then the previous argument fails because $\triangle A_iA_{i+1}A_{i+2}$ $(i=0,1,2)$ coincides with $\triangle A_0A_{1}A_{2}$ and therefore $\mathbb{M}_i=\triangle A_0A_{1}A_{2}$. 
But if $n=3$, then the only remaining option for $k$ is $k=1$ modulo 3, and this case is impossible because in this case there is $j\ne 1,2$ such that $C_0\in B^{j}_{1} B^{j}_{2}$, and therefore $B^{j}_{1}$ is closer to $A_1$ than $B^{1}_{1}$, which is against our initial assumptions.

Consequently, $B^{4}_{0} B^{4}_1$, $B^{5}_{0} B^{5}_1$ also pass through $C_{1}$, and therefore $B^{3}_i B^{3}_{i+1}$, $B^{4}_i B^{4}_{i+1}$, and $B^{5}_i B^{5}_{i+1}$ pass through $C_{i+1}$ for all $i=0,1,\ldots,n-1$.
By Lemma 2.2,
$$
\frac{|B^3_{i}B^4_{i}|}{|B^4_{i}B^5_{i}|}>\frac{|B^3_{i+1}B^4_{i+1}|}{|B^4_{i+1}B^5_{i+1}|},
$$
for all $i=0,1,\ldots,n-1$, which is again impossible.
This shows that our assumption about the existence of 5 polygons was wrong and therefore the maximal number of polygons is 4.
\end{proof}

\begin{rem}
This result about convex planar polygons can be easily extended to convex polygons in spherical and hyperbolic geometry. Indeed, any convex spherical polygon is contained in a semisphere (see \cite{fejes}, Sect. 1.6) which can be projected centrally (also known as a gnomonic projection or rectilinear projection) to a plane resulting with convex planar polygons. Similarly, for hyperbolic geometry we can use Klein-Beltrami Model (see \cite{klein}, Sect. 9.6, Fig. 226, 227). 
\end{rem}

\begin{figure}
\centering
\begin{minipage}{.5\textwidth}
  \centering
  \includegraphics[width=1\linewidth]{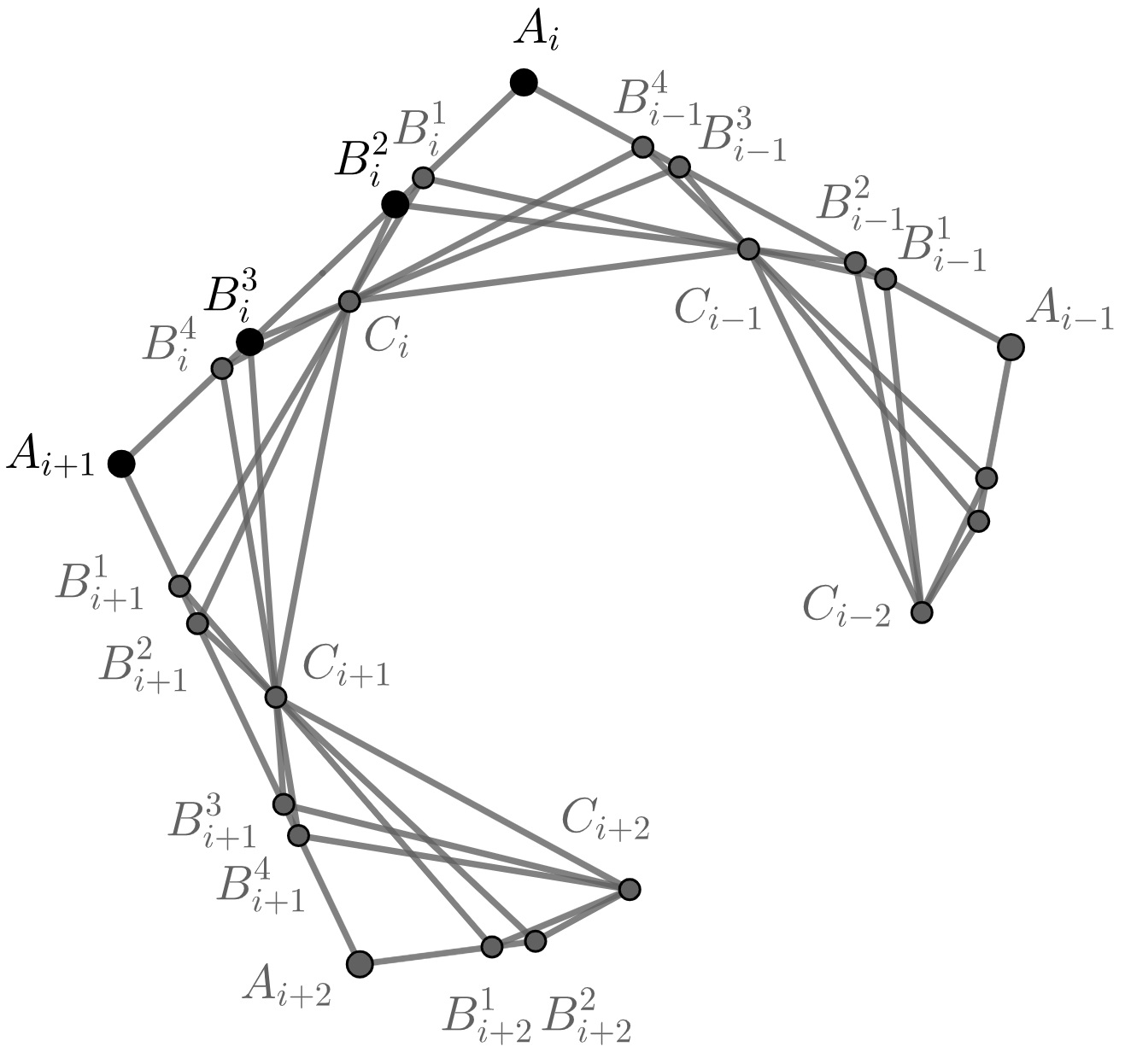}
  \captionof{figure}{Proof of Thm 2.3.}
  \label{fig7}
\end{minipage}%
\begin{minipage}{.5\textwidth}
  \centering
  \includegraphics[width=0.8\linewidth]{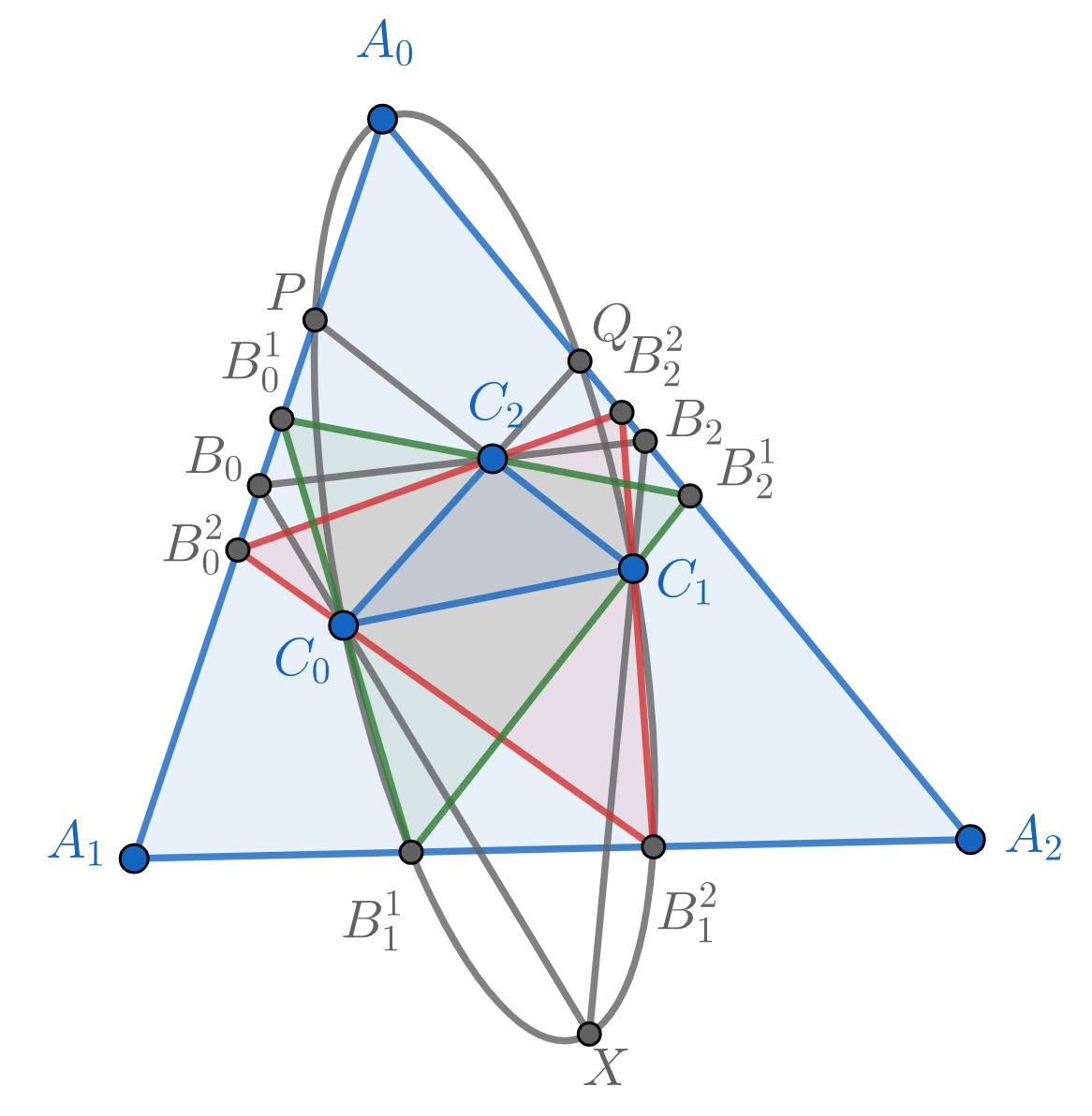}
  \captionof{figure}{\newline Maclaurin-Braikenridge's method.}
  \label{fig8}
\end{minipage}
\end{figure}

\section{Main results} We are now ready to answer the question about the maximum of the number of graphs inscribed into the graph of a convex polyhedron and circumscribed around another graph. As was mentioned before any such graph gives pattern of inscribed and circumscribed polygons on the faces of $\alpha$. By Theorem 2.2, their number can not exceed 4, which is the maximum of the number of convex polygons inscribed into one and circumscribed around another convex polygon. On the other hand we can give an example of a convex polyhedron $\alpha$ with 4 inscribed graphs $\beta$ all of which are circumscribed around the same graph $\gamma$. As $\alpha$ we take the graph of a regular octahedron all of the faces of which are equilateral triangles (see Firure 2). On each triangular face of the regular octahedron we construct 4 triangles as in Example 1 for $n=3$. The vertices and the sides of these $32$ triangles on $8$ faces of octahedron give graph $\beta$. The vertices of 8 inner triangles give the vertices of $\gamma$. Thus we proved the following result.

\begin{theorem}
The number of graphs $\beta$, that can be inscribed into graph $\alpha$ of a given convex polyhedron and circumscribed around around another graph $\gamma$ is at most 4.
\end{theorem}

Note that we can not apply the same method to tetrahedron and icosahedron to create another example similar to Figure 2, because at each vertex of the tetrahedron and icosahedron odd number of edges and faces meet (see Appendix C). Because of the odd number, in the resulting picture the vertices of $\beta$ will not alternate between closer and farther points on the edges of $\alpha$. Example 1 provides 4 polygons for any $n$ and therefore we can construct 4 squares or 4 pentagons as in the case of triangles. But for the same reason we can not use them on a cube and dodecahedron to create an example because again the number of edges meeting at each vertex is odd. There are of course semiregular polyhedra formed from 2 or more different regular polygons with equal sides. But in Example 2 we have shown that in a certain sense such an example with different polygons is impossible. Nevertheless this does not exclude the possibility of more general examples.

By combining several octahedra together on their triangular faces, so that any vertex belongs to only 1 or 2 octahedra, one can construct arbitrarily large polyhedra which can serve as non-convex examples for Theorem 3.1 (see Appendix C). It would be interesting to generalize these results and their proofs for higher dimensions. Beside the question about higher dimensions several other questions can be asked for further exloration.
\begin{problem}
Is it always possible to inscribe 4 convex polygons into any given convex polygon so that all 4 are circumscribed around another convex polygon?
\end{problem}

\begin{problem}
Is it always possible to inscribe 4 graphs into graph of any given convex polyhedron so that all 4 are circumscribed about another graph?
\end{problem}

\begin{problem}
Is the octahedron and its inscribed graphs in Figure 2 the only example of this kind for Platonic regular and Archimedes' semi-regular polyhedra?
\end{problem}

\begin{problem}
What can be said about the number of inscribed/circumscribed convex polyhedra obtained by the process of truncation of their vertices?
\end{problem}

\begin{problem}
What can be said about the number of face-inscribed/circumscribed convex polyhedra?
\end{problem}

\section{Maclaurin-Braikenridge's method of conic generation} 
It is well known that if triangles $A_0 A_1 A_2$ and $C_0 C_1 C_2$ are given, moving line $B_0B_2$ passes through $C_2$, where $B_0$ and $B_2$ are on lines $A_0A_1$ and $A_0A_2$, respectively, then the locus of intersection point $X$ of $B_0C_0$ and $B_2C_1$ is in general a conic (see Figure 8). See e.g. \cite{salmon}, p. 230, \cite{pamfil}, p. 22, \cite{pamfil2}, p. 97, \cite{wylie}, \cite{pedoe}, p. 328, where this result, known as Maclaurin-Braikenridge's method of conic generation, is generalized for polygons $A_0 A_1 A_2\ldots A_{n-1}$ and $C_0 C_1 C_2\dots C_{n-1}$. If $n=3$, then this conic passes through $A_0,\ C_0,\ C_1, P,$ and $Q$, where $P=A_0A_1\cap C_1C_2$ and $Q=A_0A_2\cap C_0C_2$. This conic can intersect $A_1A_2$ at maximum 2 points shown in Figure 8 as $B^1_1$ and $B^2_1$, which can be starting point for the construction of 2 triangles $B^1_0B^1_1B^1_3$ and  $B^2_0B^2_1B^2_3$ inscribed into $\triangle A_0 A_1 A_2$ and circumscribed around $\triangle C_0 C_1 C_2$. Therefore the key argument of using Lemma 2.1 in the proof of Theorem 2.3, can also be replaced by the fact that a general conic and a line can intersect at maximum 2 points.

Also, an important and natural connection to projective geometry can be mentioned. The problem can be reformulated as finding fixed points of a projective  transformation of a line. Indeed, let $f_i$ be the central projection with center $C_i$ from the line $A_iA_{i+1}$ to the line $A_{i+1}A_{i+2}$. The polygon $B_0\ldots B_{n-1}$ is a solution to the problem if and  only if $f_{n-1}\circ\ldots\circ f_0(B_0) = B_0$. A projective transformation either has at  most two points or is the identity (Von Staudt's lemma, see e.g. \cite{gur}, p. 79), therefore the problem can have 0, 1, 2 or  $\infty$ many solutions when the shift $k$ is fixed. Thus the proof of Theorem 2.2  actually shows that this projective transformation is not the identity, under certain assumptions on the position of lines and projection centers. In any case, it is not difficult to see that the projective transformation of $A_0A_1$ obtained by taking these central projections cannot be an identity because $A_0$ is not a fixed point. If it were, then $C_n$ would have to be on the line $A_0A_n$ since $B_n$ is on that line and $B_nC_n$ must go through $A_0$. Note the fact that 2 of the 4 polygons correspond to one projective transformation and the other 2 correspond to another one. There is a shift in the index which was mentioned in Section 2.

The generalization of Maclaurin-Braikenridge's method was studied through analytic tools in \cite{poncelet}, Tome I, p. 308-314 (see \cite{brian} for geometric consideration). The proof of case $n=3$ is the converse of Pascal's theorem \cite{coxeter}, which states that hexagon $A_0PC_1XC_0Q$ is inscribed into a conic (\cite{turnbull}, p.321, \cite{tweddle}, p. 377). The general case can be proved by induction on $n$ (see Section 5). Using this general result it is possible to show that if the condition of convexity is lifted and the polygons are extended to include also the lines containing the sides of the polygons, then in the general case, the maximal number of polygons inscribed into one polygon and circumscribed around another one, can be as large as $n!(n-1)!$ (excuding degenerate cases when the conics can coincide with the lines giving infinitely many such polygons \cite{poncelet}, Tome 2, p. 10, \cite{pamfil2}, p. 97). Indeed, there are $(n-1)!$ cyclic permutations of the sides of polygon $A_0 A_1 A_2\ldots A_{n-1}$. Between these sides the vertices of polygon $C_0 C_1 C_2\dots C_{n-1}$ can be put in $n!$ ways. Because of independence, we multiply them to obtain $n!(n-1)!$, then divide by 2 (orientation is not important) and finally multiply by 2 (each conic gives at most two polygons) to obtain again $n!(n-1)!$. For example, if $n=3$, then there are at most 12 triangles, all of which are shown in Fig. 7.

\begin{figure}
\centering
\begin{minipage}{.5\textwidth}
  \centering
  \includegraphics[width=1\linewidth]{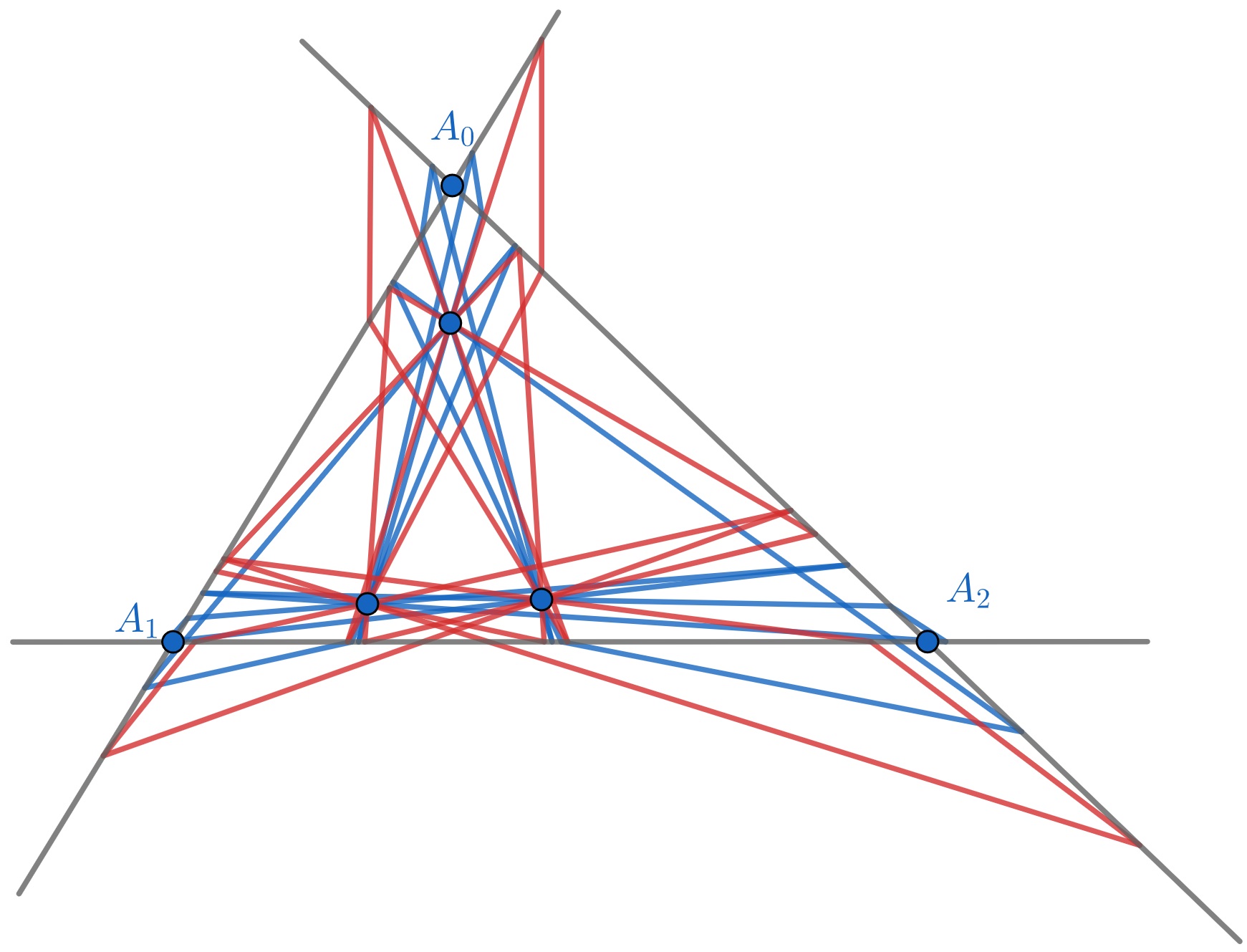}
  \captionof{figure}{There are $3!2!=12$ triangles (\url{https://www.geogebra.org/geometry/zhjwyxud}).}
  \label{fig9}
\end{minipage}%
\begin{minipage}{.5\textwidth}
  \centering
  \includegraphics[width=0.8\linewidth]{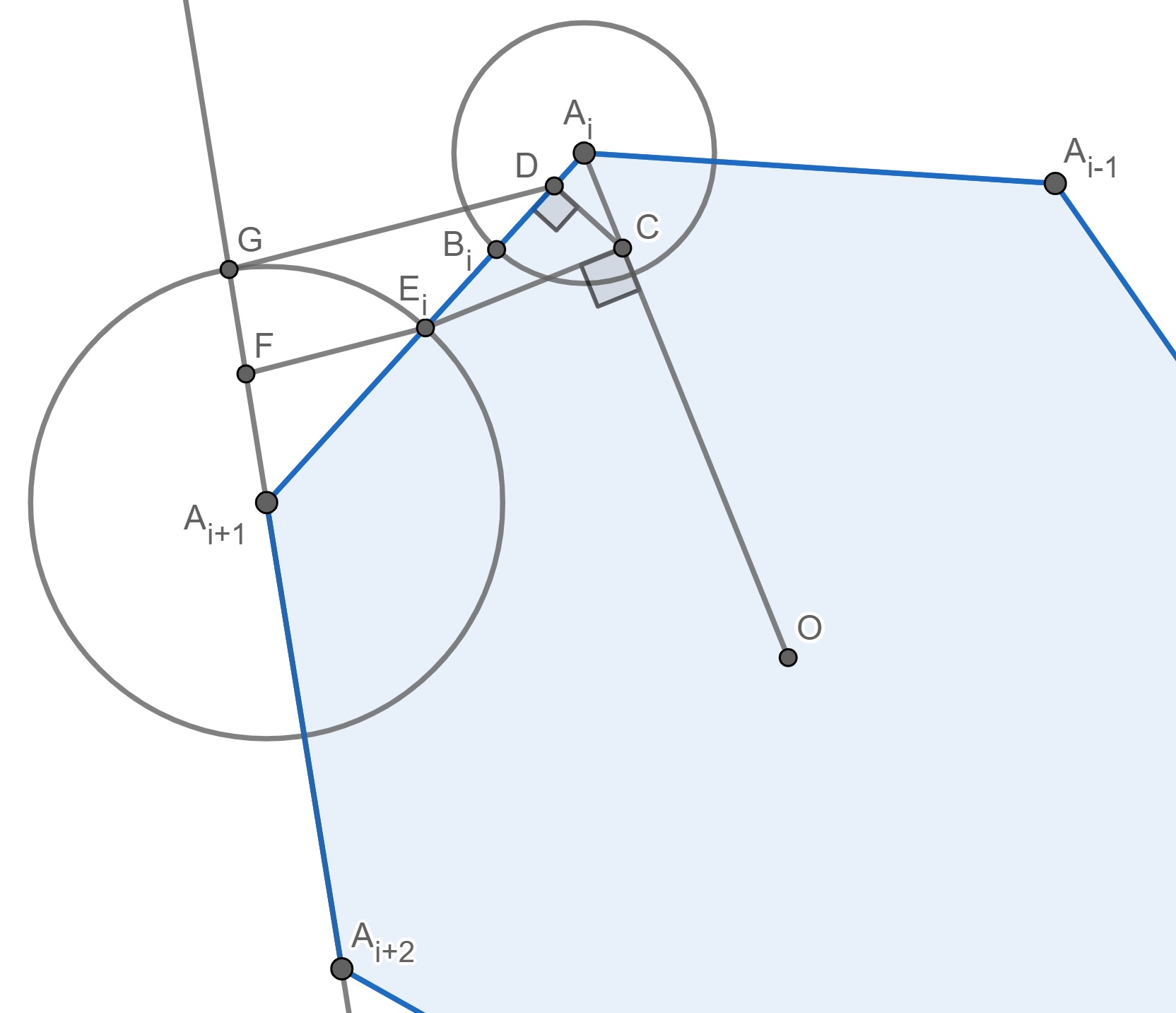}
  \captionof{figure}{\newline Construction of $B_i$ (\url{https://www.geogebra.org/geometry/jjxzgu7a}).}
  \label{fig10}
\end{minipage}
\end{figure}

\section*{Appendix A: Construction with straightedge and compass} Suppose that regular polygon $A_0 A_1 A_2\ldots A_{n-1}$ is given. Note that point $B_i$ corresponding to $x=\frac{a}{2}$ is the midpoint of $E_i A_{i}$, which is easy to construct using unmarked ruler and compass.
Below we will show how to construct using the same instruments point $B_i$ corresponding to $x=\frac{a}{1+\cos^2{\frac{\pi}{n}}}$. If $n=4$, then $x=\frac{2}{3}a$, and its construction is straightforward. So, we assume that $n\ne 4$.
Let $E_i$ be, as before, the midpoint of side $A_i A_{i+1}$ of regular polygon $A_0 A_1 A_2\ldots A_{n-1}$.
\begin{enumerate}
\item Drop perpendicular $E_iC$ to line $OA_i$.
\item Drop perpendicular $CD$ to line $A_iA_{i+1}$.
\item Draw circle with center $A_{i+1}$ through $E_i$. Let this circle intersect line $A_{i+2}A_{i+1}$ outside of line segment $A_{i+2}A_{i+1}$ at point $G$.
\item Draw line through $E_i$ parallel to $DG$. Let this line intersect line $A_{i+2}A_{i+1}$ at point $F$.
\item Draw circle with center $A_{i}$ and radius equal to through $A_{i+1}F$. This circle intersect line segment $A_{i}A_{i+1}$ at point $B_i$ such that $A_iB_i=\frac{A_iA_{i+1}}{2\left(1+\cos^2{\frac{\pi}{n}}\right)}$.
\end{enumerate}
The proof follows from the following arguments. Note first that $|A_iC|=a\cos{\frac{\pi}{n}}$, where as before $a=|E_iA_{i}|$. Then $|E_iD|=a\cos^2{\frac{\pi}{n}}$ and therefore $|A_{i+1}D|=a\left(1+\cos^2{\frac{\pi}{n}}\right)$. Since $|A_{i+1}E_i|=a$,$|A_{i+1}G|=a$.  Also from similarity of $\triangle A_{i+1}E_iF$ and $\triangle A_{i+1}DG$ we obtain that $\frac{|A_{i+1}E_i|}{|A_{i+1}D|}=\frac{|A_{i+1}F|}{|A_{i+1}G|}$, which simplifies to $|A_{i+1}F|=\frac{a}{1+\cos^2{\frac{\pi}{n}}}$. If $n=3$, then there is a shorter construction based on the fact that the corresponding sides of triangles $B_i^1B_{i+1}^1B_{i+2}^1$ and $B_{i+1}^3B_{i+2}^3B_{i+3}^3$ are parallel (see Figure 2). Similarly, if $n=6$, then one can use the fact that $|A_iD|=\frac{1}{8}\cdot |A_iA_{i+1}|$.

Note also that if regular polygon $A_0 A_1 A_2\ldots A_{n-1}$ and point $B_i$ with specific $x=x_0$ $(0<x_0<a,\ x_0\ne \frac{a}{1+\cos{\frac{\pi}{n}}})$ is given, then one can construct using the same instruments another point $B^{\prime}_i$ with $x=x_1=|A_iB^{\prime}_i|$, such that $
f(x_0)=f(x_1).
$ Indeed, $x_1$ is the second solution of quadratic equation $
\frac{x(a-x)\sin{\frac{2\pi}{n}}}{2a-2x\sin^2{\frac{\pi}{n}}}=f(x_0)
$ (the first solution is $x=x_0$), all of the coefficients of which are constructible. For $x_0= \frac{a}{1+\cos{\frac{\pi}{n}}}$, one can check that $x_1=x_0$.

\section*{Appendix B: Generalized Maclaurin-Braikenridge's conic generation}

In this part of the paper an independent proof of generalized Maclaurin-Braikenridge's theorem for $n$-gons $(n\ge 3)$ will be given. It is based on the method of mathematical induction for $n$. First, we need to prove the following lemma.

\begin{lemma}
Let fixed points $B$ and $F$, and moving point $E$ be on a given conic $\omega$. Let $P$ and $l$ be a given point and a given line on the plane of $\omega$. Let $FE$ intersect $l$ at point $K$, and $PK$ intersect $BE$ at $T$. Then the locus of point $T$ is a conic passing through points $P$ and $B$ (see Fig. 9).
\end{lemma}

\begin{figure}
\centering
\begin{minipage}{.5\textwidth}
  \centering
  \includegraphics[width=0.8\linewidth]{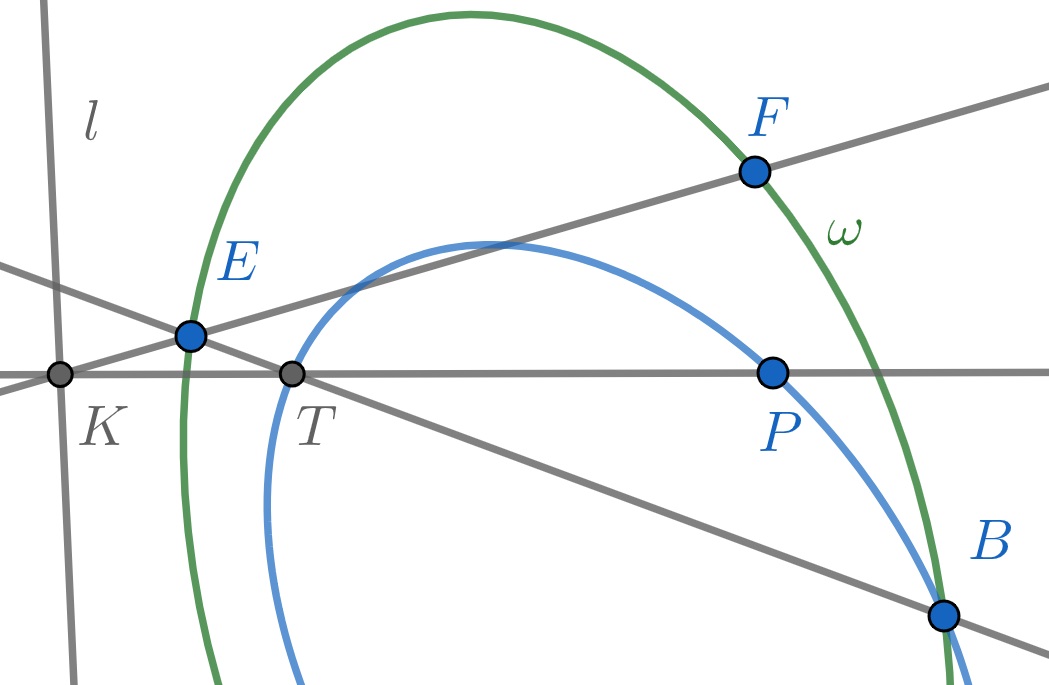}
  \captionof{figure}{Lem. 5.1.}
  \label{fig11}
\end{minipage}%
\begin{minipage}{.5\textwidth}
  \centering
  \includegraphics[width=0.8\linewidth]{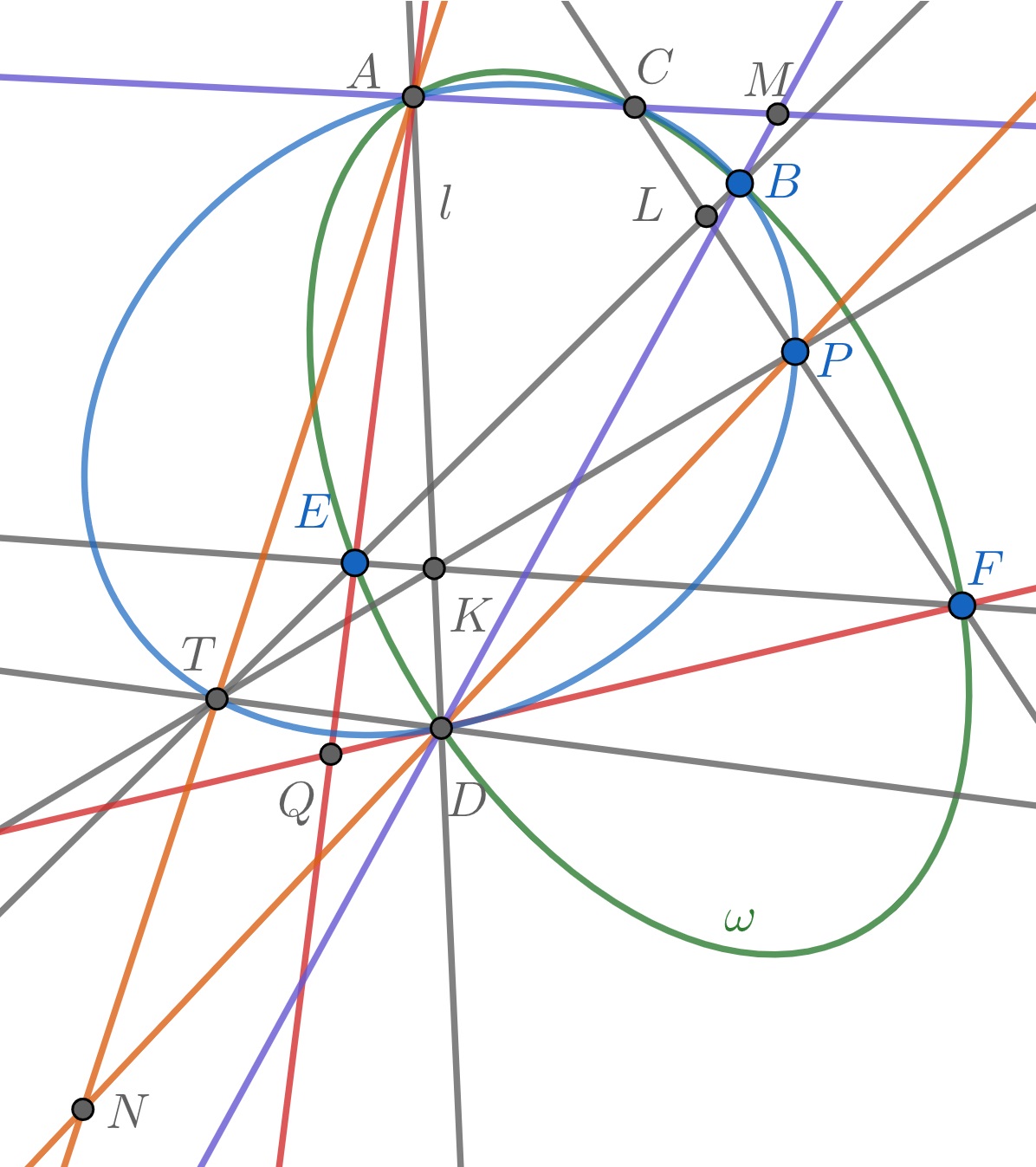}
  \captionof{figure}{Proof.}
  \label{fig12}
\end{minipage}
\end{figure}

\begin{proof}
Consider 2 cases.

Case 1. Suppose that line $l$ does not intersect conic $\omega$. Then we can apply a projective transformation mapping $l$ to infinite line and then an affine transformation mapping the image of conic $\omega $ to a circle. Let us use the same letters for the new objects. When $E$ changes on circle $\omega$, $\angle FEB$ does not change. Since $K$ is now an infinite point on infinite line $l$, $PK\parallel FK$, and therefore, $\angle PTB=\angle FEB=\textnormal{const}$. So, locus of $T$ is a circle. This means that inverse of affine and projective transformations of the locus of $T$ will be a conic.

Case 2. Suppose now that $l$ intersects $\omega$ at points $A$ and $D$ as in Fig. 10. Denote $C=PF\cap \omega$, 
$L=PF\cap BE$, $M=AC\cap BD$, $Q=AE\cap FD$, and $N=AT\cap PD$. By Pascal's theorem, since hexagon $ACFDBE$ is inscribed into conic $\omega$, points $M,L,Q$ are collinear. By Desarques' theorem, since $\triangle AET$ and $\triangle DFP$ are perspective with respect to point $K$ ($AD$, $EF$, and $TP$ are concurrent), these triangles are also perspective with respect to a line. Therefore points $M$, $L$, and $N$ are collinear. Then by Pascal's theorem, hexagon $ACPDBT$ is inscribed into a conic. Since points $A,\ C,\ P,\ D,\ B$ are fixed, the locus of point $T$ is a conic.
\end{proof}

We are ready to state and prove the following general result.
\begin{theorem}
Let $l_0, l_1,\ldots ,l_{n-1}$ be $n$ lines, and $C_0, C_1,\ldots ,C_{n-1}$ be $n$ points in general position $(n\ge 3)$. Let changing points $B_i\in l_i$ $(i=0,1,\ldots,n-1)$ and $B_n\in l_0$ be chosen so that $C_i\in B_iB_{i+1}$ $(i=0,1,\ldots,n-1)$. Then the locus of point $X_n=B_0 B_1\cap B_n B_{n-1}$ is a conic passing through $C_0$ and $C_{n-1}$.
\end{theorem}

\begin{proof}
We saw in Section 4 that the case $n=3$ is the converse of Pascal's theorem and therefore it is easily proved \cite{coxeter2}, p. 85. Suppose that the claim is true for $n=k$ $(k\ge 3)$, which means that the locus of $X_k=B_0 B_1\cap B_{k} B_{k-1}$ is a conic $\omega_k$ passing through $C_0$ and $C_{k-1}$. If we take in Lemma 5.1 $E=X_{k-1}$, $F=C_0$, $B=C_{k-1}$, $K=B_k$, $\omega=\omega_k$, $P=C_k$, $T=X_k$, $l=l_k$, then we obtain that the locus of $X_k$ is a conic passing through $C_0$ and $C_{k}$.
\end{proof}

As in Section 4 above Projective geometry provides a two-lines but non-elementary proof of the generalized
Maclaurin-Braikenridge conic generation. The map sending the
line $B_0B_1$ to the line $B_{n-1}B_n$ in Theorem 5.2 is a projective map between
two pencils of lines, and by the Steiner’s conic generation (see e.g. \cite{gur}, p. 178) the intersection
point of such projectively related lines traces a conic. 

\section*{Appendix C: Other Platonic solids and non-convex polyhedra.} Figures 13-16 show 4 remaining Platonic solids which enjoy many symmetries but do not give an example for Theorem 3.1 because of odd number of faces meeting at each vertex. It would be interesting to determine if for 4 dimensions analogous regular polytopes give an example of 4 inscribed/circumscribed graphs or not.

\begin{figure}
\centering
\begin{minipage}{.5\textwidth}
  \centering
  \includegraphics[width=1\linewidth]{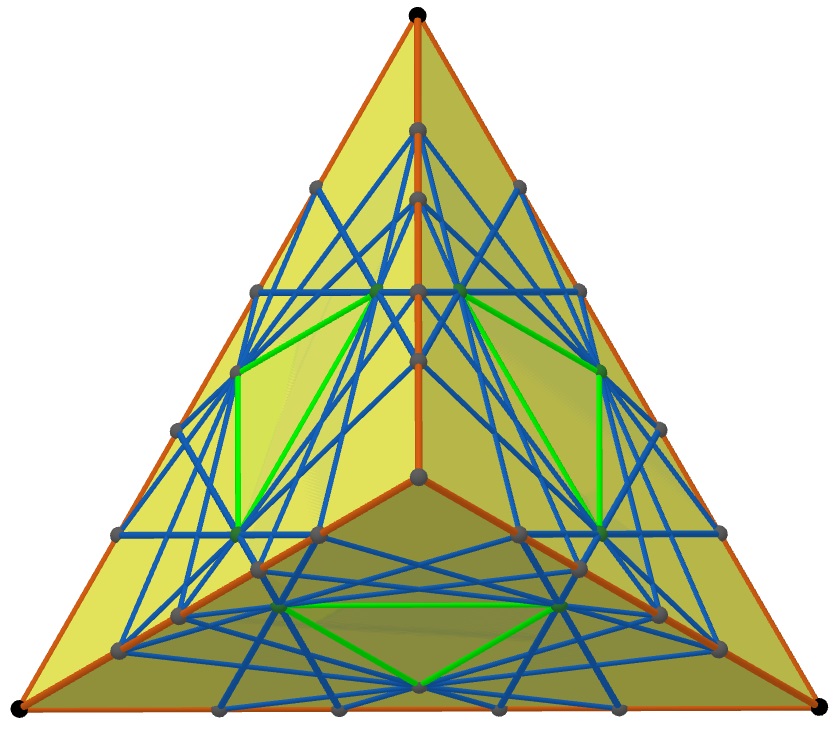}
  \captionof{figure}{Tetrahedron (\url{https://www.geogebra.org/3d/uptbuvkr}).}
  \label{fig13}
\end{minipage}%
\begin{minipage}{.5\textwidth}
  \centering
  \includegraphics[width=1\linewidth]{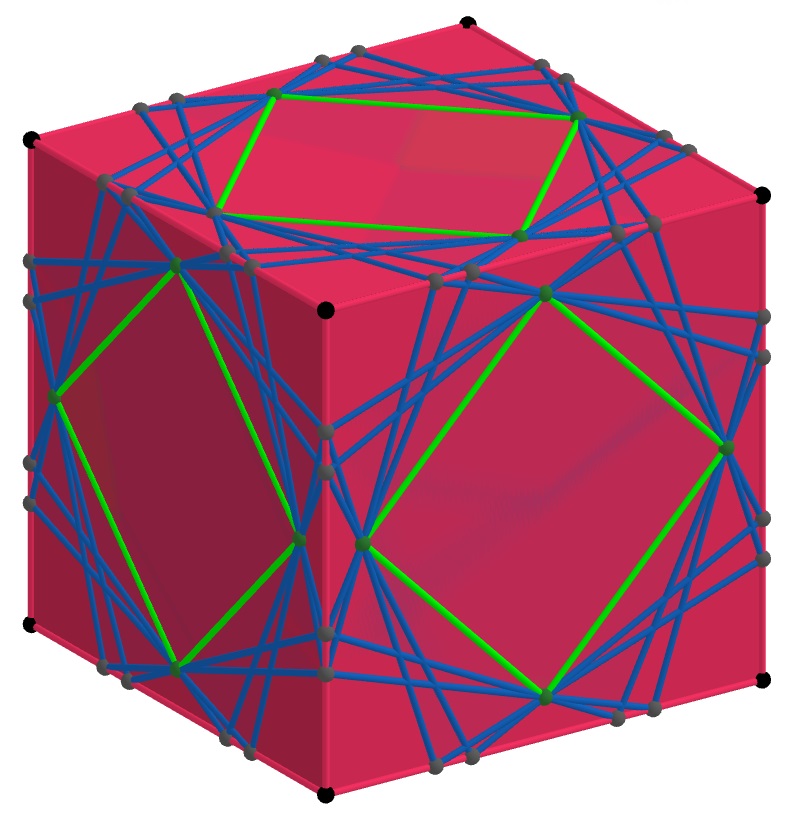}
  \captionof{figure}{Cube (\url{https://www.geogebra.org/3d/jszs6rbw}).}
  \label{fig14}
\end{minipage}
\end{figure} 

\begin{figure}
\centering
\begin{minipage}{.5\textwidth}
  \centering
  \includegraphics[width=1\linewidth]{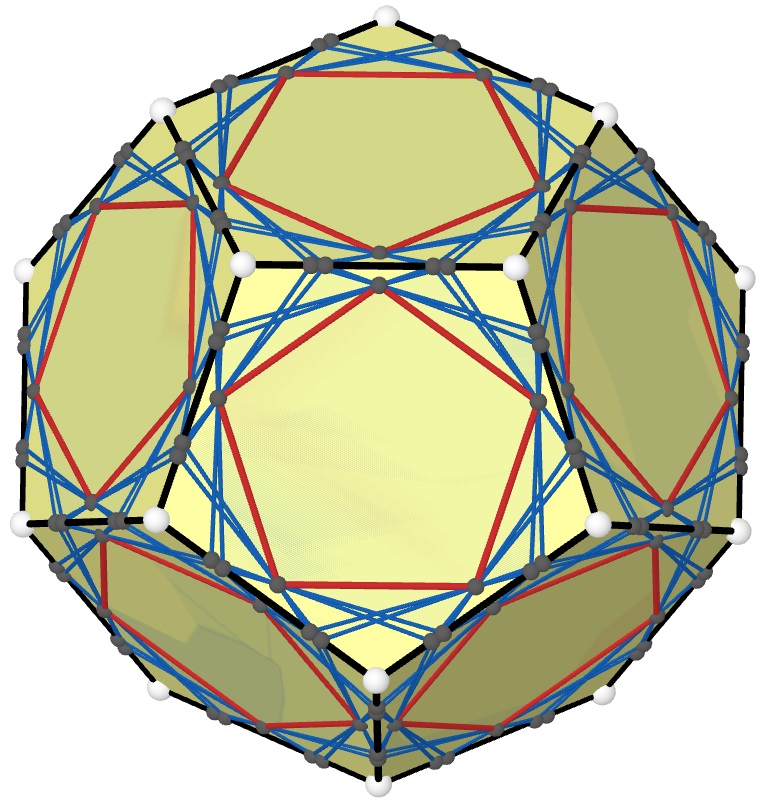}
  \captionof{figure}{Dodecahedron (\url{https://www.geogebra.org/3d/s8synt25}).}
  \label{fig15}
\end{minipage}%
\begin{minipage}{.5\textwidth}
  \centering
  \includegraphics[width=1\linewidth]{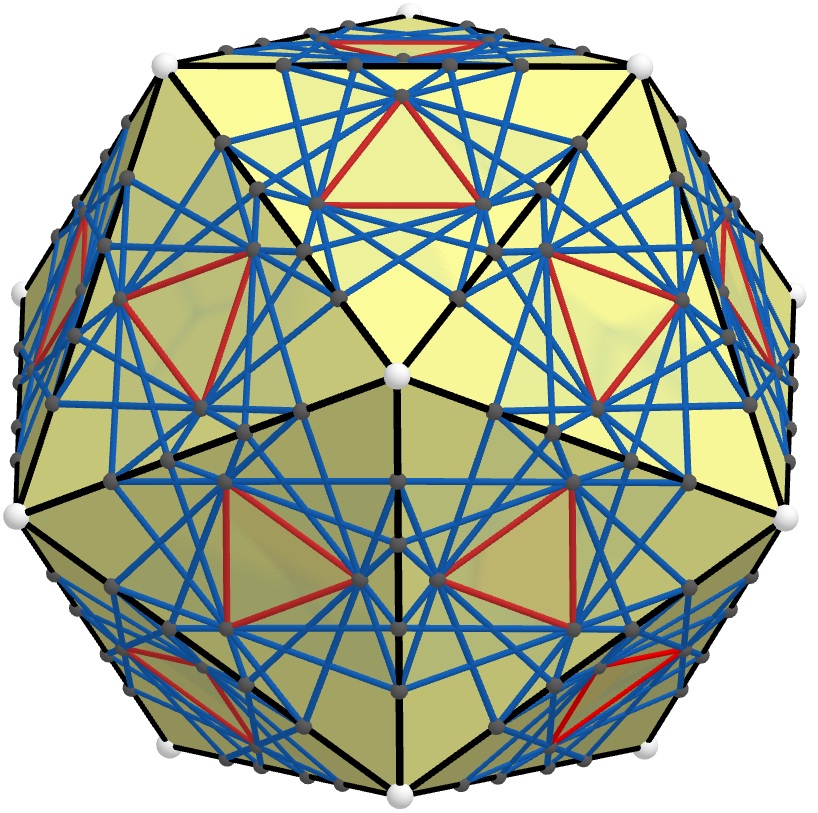}
  \captionof{figure}{Icosahedron (\url{https://www.geogebra.org/3d/vte3rcss}).}
  \label{fig16}
\end{minipage}
\end{figure} 

There are also non-convex polyhedra for which it is possible to construct an example similar to Figure 2. In Figure 17 several octahedra are joined along their triangular faces so that at each vertex only 4 or 6 faces meet. One can increase the number of octahedra to obtain arbitrarily large such non-convex examples. In Figure 18 and Figure 19, you can see models of non-convex polyhedra with 10 vertices, 24 edges, and 16 triangular faces, and 11 vertices, 27 edges, and 18 triangular faces, respectively, made of Geomag pieces. Note that at each vertex of these polyhedra again 4 or 6 faces meet, which make it possible to construct examples similar to the octahedron in Figure 2.

\begin{figure}
\centering

  \includegraphics[width=0.9\linewidth]{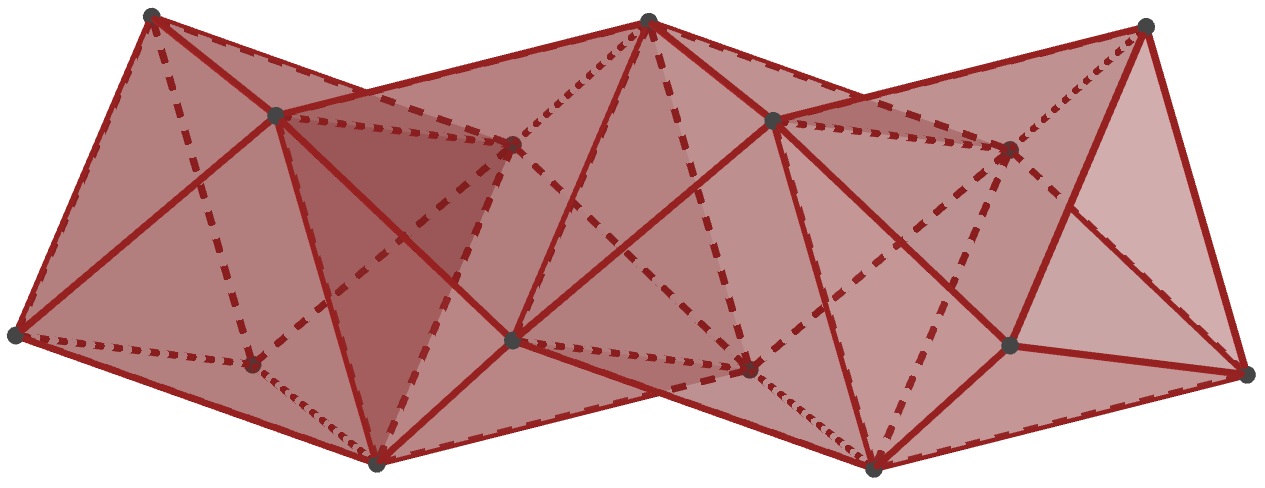}
  \captionof{figure}{Non-convex poyhedron obtained by gluing together several octahedra so that at any vertex at most 2 octahedra can meet \url{https://www.geogebra.org/3d/aru7qqux}.}
  \label{fig17}

\end{figure}
\begin{figure}
\centering
\begin{minipage}{.5\textwidth}
  \centering
  \includegraphics[width=1\linewidth]{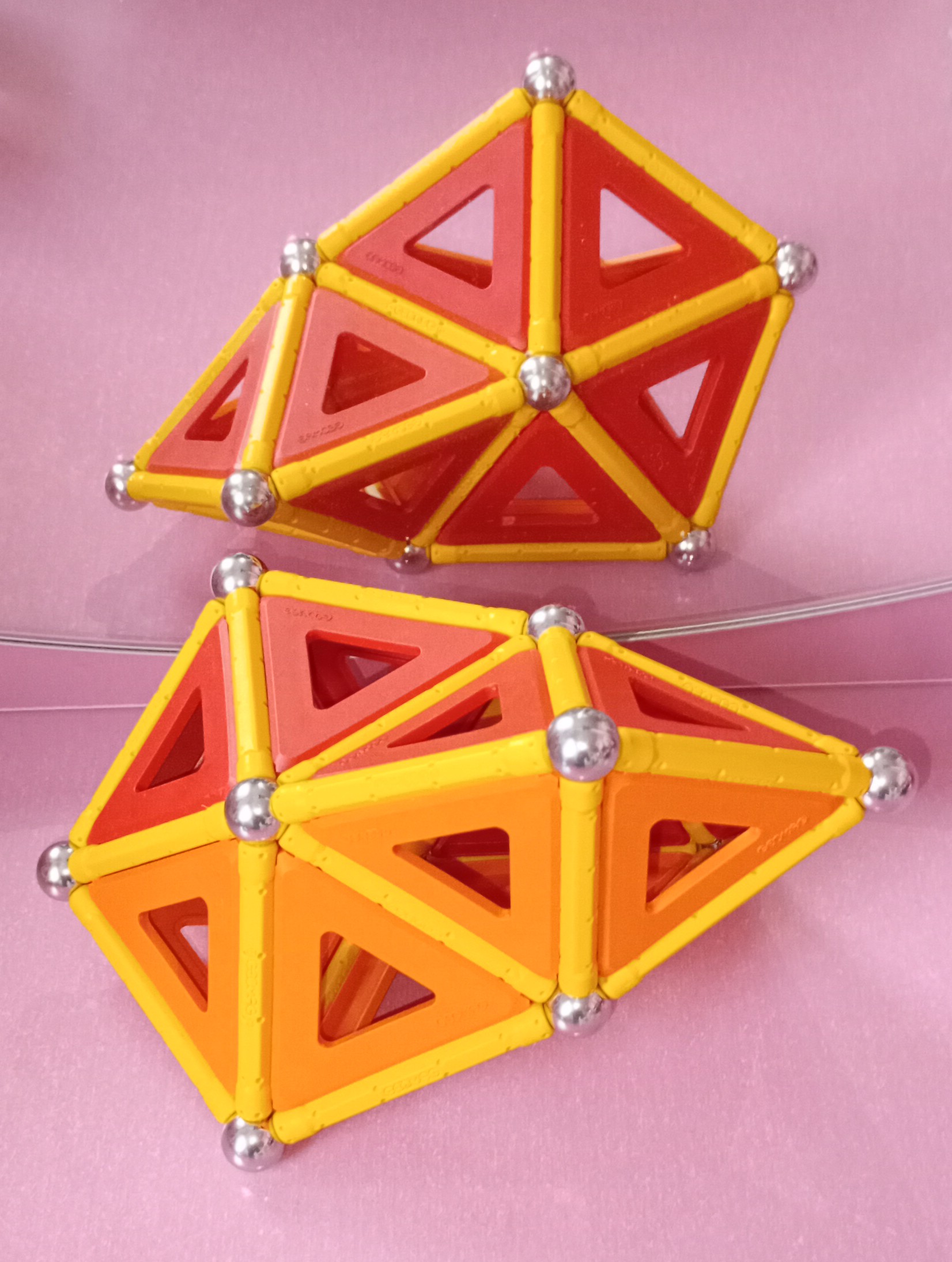}
  \captionof{figure}{Non-convex poyhedron with even number of identical faces meeting at each vertex.}
  \label{fig18}
\end{minipage}%
\begin{minipage}{.5\textwidth}
  \centering
  \includegraphics[width=0.77\linewidth]{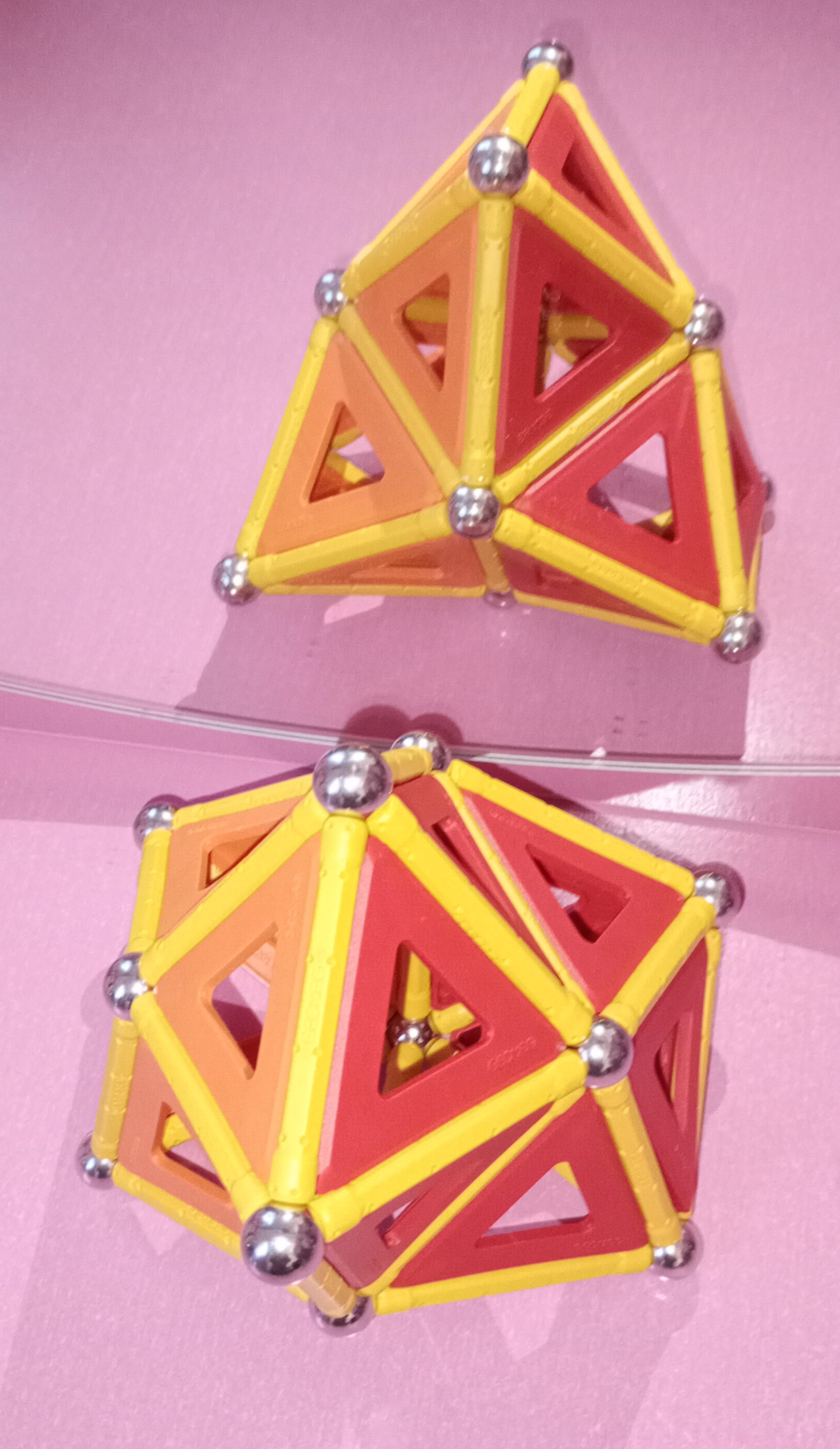}
  \captionof{figure}{As in the left example, the number of faces meeting at each vertex is either 4 or 6.}
  \label{fig19}
\end{minipage}
\end{figure}

Planar grids formed by regular hexagons, squares and triangles can be interpreted as infinitely large (degenerate) polyhedra (see \cite{fejes}, Sect. 1.7). In the case of hexagons (Figure 20) odd number of faces meet at each vertex, but for squares (Figure 21) and triangles this number is even. For this reason, hexagons do not give examples similar to Fig. 2, but squares and triangles do.

\begin{figure}
\centering\begin{minipage}{.5\textwidth}
  \centering
  \includegraphics[width=1\linewidth]{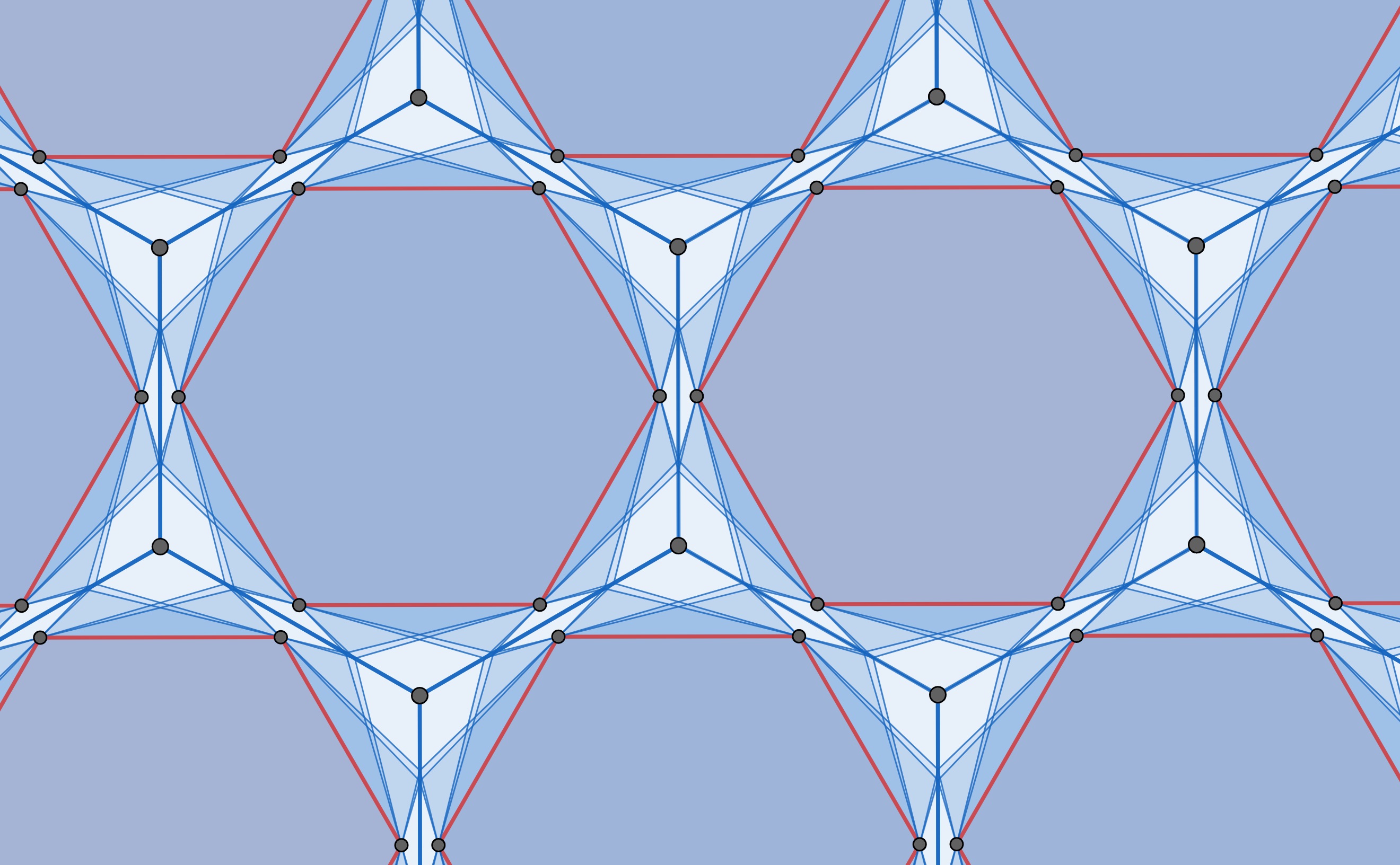}
  \captionof{figure}{Odd number of hexagonal faces meeting at each vertex do not give an example.}
  \label{fig20}
\end{minipage}%
\begin{minipage}{.5\textwidth}
  \centering
  \includegraphics[width=0.9\linewidth]{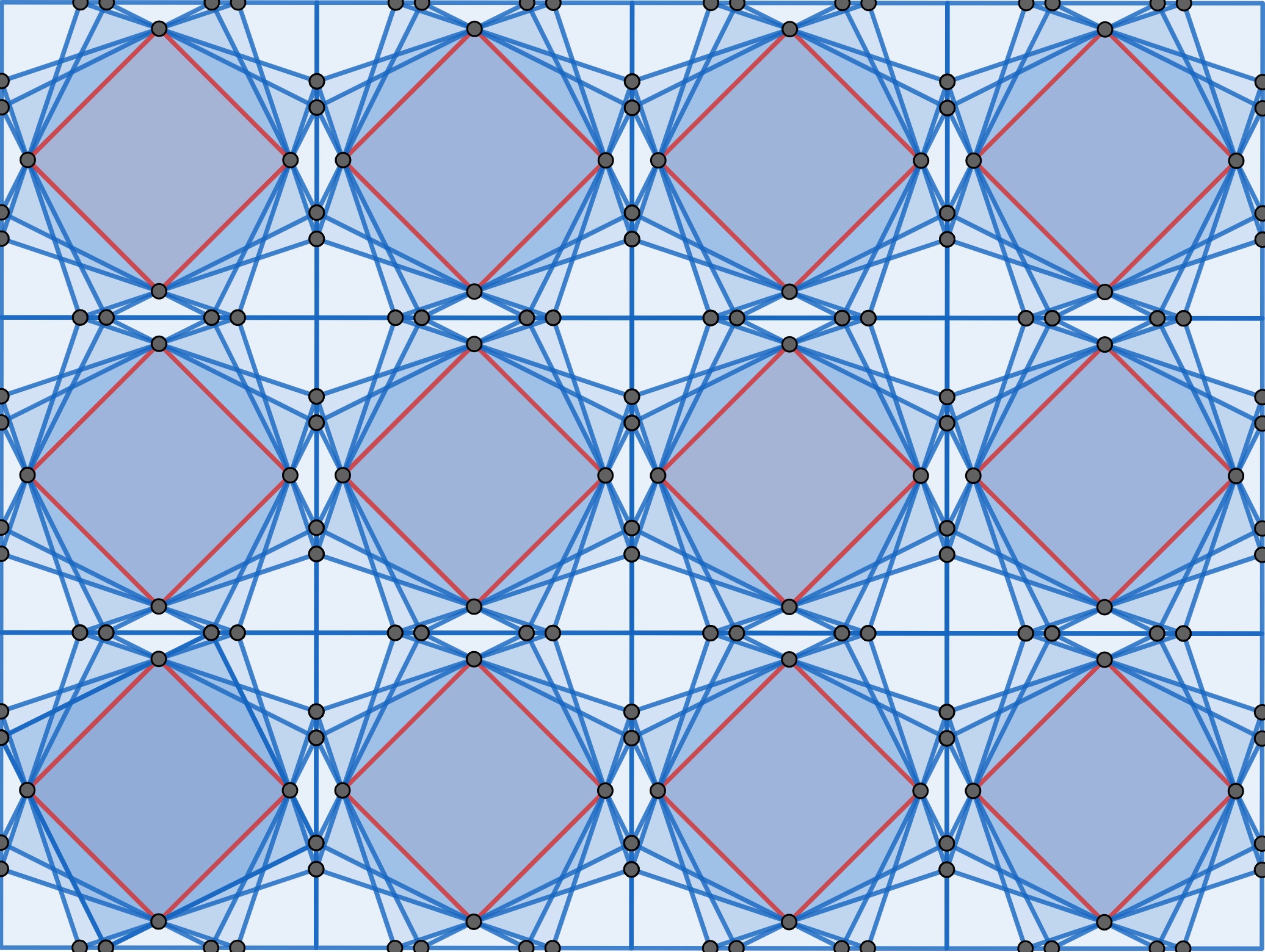}
  \captionof{figure}{Even number of faces meeting at each vertex give an example.}
  \label{fig21}
\end{minipage}

\end{figure}

\section{Conclusion}
In the paper the question about the maximum of the number of graphs inscribed into the graph of a convex polyhedron and circumscribed about another graph, is discussed. The maximal number of such graphs is shown to be 4. Constructible examples of these 4 graphs (convex polygons) in the case of regular icosahedron (regular polygon) are given. The special problem for convex polygons on a plane is also solved. The connection with projective geometry, generalized Maclaurin-Braikenridge's conic generation method and its new proof based on mathematical induction are also included.

\section*{Acknowledgments}
%This work was supported by ADA University Faculty Research and Development Fund.

\section{Declarations}
\textbf{Ethical Approval.}
Not applicable.
 \newline \textbf{Competing interests.}
None.
  \newline \textbf{Authors' contributions.} 
Not applicable.
  \newline \textbf{Funding.}
This work was completed with the support of ADA University Faculty Research and Development Fund.
  \newline \textbf{Availability of data and materials.}
Not applicable
% ------------------------------------------------------------------------

% ------------------------------------------------------------------------
\end{document}